\documentclass[11pt]{amsart}
\usepackage{calc,amssymb,amsmath,ulem}
\usepackage{alltt}
\RequirePackage[dvipsnames,usenames]{color}

\input{kmacros3.sty}
\input{xy}
\input{chancery.sty}
\xyoption{all}
\usepackage[backref=page]{hyperref}
\usepackage{xcolor, enumitem}
\hypersetup{
  pdfauthor={Patrick Graf and Karl Schwede},
  pdftitle={Discreteness of F-jumping numbers at isolated non-Q-Gorenstein points},
  pdfkeywords={Test ideals, F-jumping numbers, Q-Gorenstein, multiplier ideals},
  pdfstartview={Fit},
  pdfpagelayout={OneColumn},
  pdfpagemode={UseNone},
  linktoc=all,
  breaklinks,
  colorlinks,
  linkcolor={blue!80!black},
  citecolor={green!50!black},
  urlcolor={red!80!black}
}
\usepackage{graphicx}
\usepackage[all,cmtip]{xy}
\usepackage{verbatim,xspace}
\usepackage[left=1.25in,top=1.09in,right=1.25in,bottom=1.09in]{geometry}
\usepackage{mabliautoref}

\renewcommand{\O}{\sO}

\renewcommand{\to}[1][]{\xrightarrow{\ #1\ }}

\newcommand{\into}[1][]{\lhook \joinrel \xrightarrow{\ #1\ }}

\renewcommand{\sHom}[0]{{\mcH\mco\mcm}}

\renewcommand{\myR}{\mcR}
\def\frm{\mathfrak{m}}

\def\rd#1.{\lfloor{#1}\rfloor}
\def\rp#1.{\lceil{#1}\rceil}

\newcommand{\wb}{\overline}
\renewcommand{\wt}{\widetilde}
\newcommand{\ntl}{\natural}
\DeclareMathOperator{\tr}{tr}

\renewcommand\factor[2]{\left. \raise 2pt\hbox{$#1$} \right/\hskip -2pt \raise -2pt\hbox{$#2$}}

\setlist[enumerate]{label=(\thetheorem.\arabic*), before={\setcounter{enumi}{\value{equation}}}, after={\setcounter{equation}{\value{enumi}}}}

\numberwithin{equation}{theorem}

\begin{document}

\title[$F$-jumping numbers at isolated non-$\bQ$-Gorenstein points]{Discreteness of $F$-jumping numbers at isolated non-$\bQ$-Gorenstein points}
\date{April 11, 2017}
\thanks{The first-named author was supported in part by the DFG grant ``Zur Positivit\"at in der komplexen Geometrie''. The second named author was supported in part by the NSF grant DMS \#1064485,
  NSF FRG Grant DMS \#1501115, NSF CAREER Grant DMS \#1501102}

\begin{abstract}
We show that the $F$-jumping numbers of a pair $(X, \fra)$ in positive characteristic have no limit points whenever the symbolic Rees algebra of $-K_X$ is finitely generated outside an isolated collection of points.
We also give a characteristic zero version of this result, as well as a generalization of the Hartshorne--Speiser--Lyubeznik--Gabber stabilization theorem describing the non-$F$-pure locus of a variety.
\end{abstract}

\author[P.~Graf]{Patrick Graf}
\address{PG: Lehrstuhl f\"ur Mathematik I, Universit\"at Bay\-reuth,
  95440 Bayreuth, Germany} %
\urladdr{\href{http://www.pgraf.uni-bayreuth.de/en/}{www.pgraf.uni-bayreuth.de/en/}}
\email{\href{mailto:patrick.graf@uni-bayreuth.de}{patrick.graf@uni-bayreuth.de}}

\author[K.~Schwede]{Karl Schwede}
\address{KS: Department of Mathematics, The University of Utah,
155 S 1400 E Room 233, Salt Lake City, UT 84112, USA}
\urladdr{\href{http://www.math.utah.edu/~schwede/}{www.math.utah.edu/~schwede/}}
\email{\href{mailto:schwede@math.utah.edu}{schwede@math.utah.edu}}

\subjclass[2010]{13A35, 14F18}

\keywords{Test ideals, $F$-jumping numbers, $\bQ$-Gorenstein, multiplier ideals}

\maketitle

\section{Introduction}

By now it is well understood that there is an interesting connection between multiplier ideals in characteristic zero, defined via resolution of singularities, and test ideals in positive characteristic, defined via the behavior of the Frobenius map.
Recall that for any complex pair $(X, \fra)$, the multiplier ideal $\mJ(X, \fra^t)$ gets smaller as $t$ increases, but it does not change if we increase $t$ just slightly: $\mJ(X, \fra^t) = \mJ(X, \fra^{t+\varepsilon})$ for $0 < \varepsilon \ll 1$.
Hence it makes sense to define the jumping numbers of $(X, \fra)$ as those real numbers $t_i$ such that $\mJ(X, \fra^{t_i}) \subsetneq \mJ(X, \fra^{t_i-\varepsilon})$ for $\varepsilon > 0$.
By analogy, the $F$-jumping numbers are the real numbers $t_i$ where the test ideal jumps or changes: \mbox{$\tau(X, \fra^{t_i}) \subsetneq \tau(X, \fra^{t_i-\varepsilon})$} for $\varepsilon > 0$.

The discreteness and rationality of ($F$-)jumping numbers has been studied by many authors, e.g.~\cite{EinLazSmithVarJumpingCoeffs,BoucksomdeFernexFavreUrbinatiValSpaces,HaraMonskyFPureThresholdsAndFJumpingExponents,BlickleMustataSmithDiscretenessAndRationalityOfFThresholds,BlickleMustataSmithFThresholdsOfHypersurfaces,KatzmanLyubeznikZhangOnDiscretenessAndRationality,BlickleSchwedeTakagiZhang,KatzmanZhangCastelnuovoRegularity,SchwedeTuckerTestIdealsOfNonPrincipal,KatzmanSchwedeSinghZhang}.
In characteristic zero, discreteness and rationality of jumping numbers is elementary if $X$ is $\bQ$-Gorenstein, but rationality fails in general~\cite[Theorem~3.6]{UrbinatiDiscrepanciesOfNonQGorensteinVars}. Discreteness remains an open problem, with several special cases known, e.g.~if the non-$\bQ$-Gorenstein locus of $X$ is zero-dimensional~\cite[Theorem~1.4]{GrafJumpingCoefficientsOfNonQGor} (and \cite[Theorem~5.2]{UrbinatiDiscrepanciesOfNonQGorensteinVars} for an earlier, weaker version).
For test ideals, discreteness and rationality are known whenever the algebra of local sections
\[ \sR \big( X, -(K_X + \Delta) \big) := \bigoplus_{m \ge 0} \O_X \big( \rd -m(K_X + \Delta). \big) \]
(also known as the symbolic Rees algebra) is finitely generated~\cite{BlickleSchwedeTakagiZhang, SchwedeDiscretenessQGorenstein, ChiecchioEnescuMillerSchwede}. In this paper, we prove the following result.

\begin{theoremA*}[Discreteness of $F$-jumping numbers, \autoref{thm.discrete isol}]
Suppose that $X$ is a normal variety over an $F$-finite field $k$ of positive characteristic and that $\Delta \geq 0$ is a $\bQ$-divisor such that $\sR \big( X, -(K_X + \Delta) \big)$ is finitely generated except at an isolated collection of points. Suppose $\fra \subseteq \O_X$ is a nonzero coherent ideal sheaf. Then the $F$-jumping numbers of $(X, \Delta, \fra)$ have no limit points.
\end{theoremA*}

The corresponding statement in characteristic zero is also new.

\begin{theoremB*}[Discreteness of jumping numbers, \autoref{thm.discrete klt}]
Suppose that $X$ is a normal variety over a field $k$ of characteristic zero and that $\Delta \geq 0$ is a $\bQ$-divisor such that $\sR \big( X, -(K_X + \Delta) \big)$ is finitely generated except at an isolated collection of points. Suppose $\fra \subseteq \O_X$ is a nonzero coherent ideal sheaf. Then the jumping numbers of $(X, \Delta, \fra)$ have no limit points.
\end{theoremB*}

We would like to point out that the question whether $F$-jumping numbers are always rational is still open. However the characteristic zero counterexample mentioned above suggests that maybe one should expect a negative answer.

The method used to prove these results builds upon \cite{UrbinatiDiscrepanciesOfNonQGorensteinVars} and \cite{GrafJumpingCoefficientsOfNonQGor}.
In particular, we prove global generation of (Frobenius pushforwards of) sheaves used to compute test ideals after twisting by a sufficiently ample divisor $H$.
If $X$ is projective, discreteness of $F$-jumping numbers follows quickly, since the twisted test ideals are globally generated by vector subspaces within the finite-dimensional vector space $H^0(X, \O_X(H))$.
The general case is easily reduced to the projective case by a compactification argument.

Using these same methods, we also obtain a generalization of the Hartshorne--Speiser--Lyubeznik--Gabber stabilization theorem.  Let us motivate the result briefly. Notice that if $R$ is a ring, we have canonical maps $\Hom_R(F^e_* R, R) \to R$ obtained by evaluation at~$1$.  These images yield a descending chain of ideals $J_e$.  If $K_R$ is Cartier, it follows from \cite{HartshorneSpeiserLocalCohomologyInCharacteristicP,LyubeznikFModulesApplicationsToLocalCohomology,Gabber.tStruc} that these images stabilize, giving a canonical scheme structure to the non-$F$-pure locus of $X = \Spec R$.
Blickle and B\"ockle also proved a related stabilization result for arbitrary rings (and even more) \cite{BlickleBockleCartierModulesFiniteness,BlickleTestIdealsViaAlgebras} but their result does not seem to imply that $J_e = J_{e+1}$ for $e \gg 0$ (Blickle obtained another result which implies stabilization of a different set of smaller ideals).  However, as a corollary of our work, we obtain the following generalization, also see \cite[Proposition 3.7]{ChiecchioEnescuMillerSchwede}.

\begin{theoremC*}[HSLG-type stabilization, \autoref{thm.HSLG}]
Suppose that $X$ is a normal variety over an $F$-finite field $k$ of characteristic $p > 0$.  Set
\[
J_e := \Image \Big(F^e_* \O_X \big( (1 - p^e) K_X \big) \cong \sHom_{\O_X}(F^e_* \O_X, \O_X) \xrightarrow{\textnormal{eval@1}} \O_X \Big) \subset \O_X.
\]
If $\sR(X, -K_X)$ is finitely generated except at an isolated collection of points, then $J_e = J_{e+1}$ for all $e \gg 0$.
\end{theoremC*}

\begin{remark}
There should be a more general version of \autoref{thm.HSLG} with $\sR \big( X, -(K_X + \Delta) \big)$ in place of $\sR(X, -K_X)$, but for the proof one would probably need to generalize \autoref{thm.AsymptoticGG} further.
\end{remark}

We end the introduction by pointing out some geometric and cohomological conditions on the singularities of $X$ which ensure that our assumptions on the anticanonical algebra of $X$ are satisfied.

\begin{proposition}[Klt or rational singularities and finite generation] \label{prp.klt fg}
Let $X$ be a normal variety over a field $k$. Assume any of the following:
\begin{enumerate}[leftmargin=*, itemsep=1ex]
\item\label{itm.klt fg.1} $\Char(k) = 0$ and there is a $\bQ$-divisor $D \ge 0$ such that the pair $(X, D)$ is klt except at an isolated collection of points.
\item\label{itm.klt fg.3} $\dim X \le 3$ and $X$ has pseudorational singularities except at an isolated collection of points.
\end{enumerate}
Then for any $\bQ$-Weil divisor $B$ on $X$, the algebra of local sections $\sR(X, B)$ is finitely generated except at an isolated collection of points.
Hence the assumptions of Theorems~\ref{thm.discrete isol},~\ref{thm.discrete klt} and~\ref{thm.HSLG} are satisfied in this case.
\end{proposition}

% Recall that a normal variety $X$ in any characteristic is said to have \emph{pseudorational singularities} if it is Cohen--Macaulay and for any proper birational map $\pi\colon Y \to X$ with $Y$ normal, we have
% \[ \pi_* \omega_Y = \omega_X \]
% where $\omega_X$ and $\omega_Y$ denote the canonical sheaves of $X$ and of $Y$, respectively.
For the definition of \emph{pseudorational singularities}, see~\cite[Section~2, p.~102]{LipmanTeissierPseudoRational}.
An equivalent condition, emphasizing the point of view of extendability of differential forms, is given in~\cite[Section~4, Corollary on p.~107]{LipmanTeissierPseudoRational}.
From the latter condition, it is easy to see that for a klt threefold pair $(X, D)$, the space $X$ has pseudorational singularities except at an isolated collection of points since $X$ is Cohen--Macaulay in codimension two anyways.

\subsection*{Acknowledgements}

The authors first began working on this project at the 2015 Summer Research Institute on Algebraic Geometry held at the University of Utah.
Furthermore they would like to thank the referee for helpful suggestions.

\section{Preliminaries}

\begin{convention}
Throughout this paper, all schemes are Noetherian and separated and of finite type over a field which in characteristic $p > 0$ is always assumed to be $F$-finite.
In characteristic $p > 0$, $F\colon X \to X$ denotes the absolute Frobenius map, acting on an affine scheme $U = \Spec R$ by $r \mapsto r^p$.
\end{convention}

The material in this section is mostly well-known to experts, and collected for convenience of the reader.

\subsection{Grothendieck duality}

We will use the following special case of Grothendieck duality~\cite[Proposition~5.67]{KollarMori}. Let $f\colon X \to Y$ be a finite map and $\sF, \sG$ coherent sheaves on $X$ and on $Y$, respectively. Then there is a natural $f_* \O_X$-linear isomorphism
\begin{equation} \label{eq.GD}
\sHom_{\O_Y}(f_* \sF, \sG) = f_* \sHom_{\O_X}(\sF, f^! \sG),
\end{equation}
where $f^! \sG := \sHom_{\O_Y}(f_* \O_X, \sG)$.
Furthermore we will use the fact that if $X$ is essentially of finite type over an $F$-finite field, then $F^! \omega_X \cong \omega_X$ where $\omega_X$ is the canonical sheaf of $X$.

Suppose now that $X$ is a normal integral scheme of finite type over an $F$-finite field of characteristic $p > 0$.
Then we have a canonical map (called the trace map)
\[ F^e_* \omega_X \to \omega_X \]
which under~\autoref{eq.GD} corresponds to $\id \in F^e_* \sHom_{\O_X}(\omega_X, \omega_X)$.
Let $K_X$ be a canonical divisor on $X$.
Twisting by $\O_X(-K_X)$ and reflexifying yields
\[ F^e_* \O_X \big( (1 - p^e) K_X \big) \to \O_X \]
and then for any effective Weil divisor $D \ge 0$ by restriction we obtain a map
\begin{equation} \label{eq.trace}
\tr\colon F^e_* \O_X \big( (1 - p^e) K_X - D \big) \to \O_X.
\end{equation}
Using~\autoref{eq.GD} again, the left-hand side sheaf is identified with $\sHom_{\O_X} \big( F^e_* \O_X(D), \O_X \big)$.
It is straightforward to check that under this identification,~\autoref{eq.trace} becomes the ``evaluation at~$1$'' map
\[ \sHom_{\O_X} \big( F^e_* \O_X(D), \O_X \big) \xrightarrow{\textnormal{eval@1}} \O_X. \]

\subsection{Test ideals}

We recall the following definition of test ideals from the literature.

\begin{definition}[\protect{\cite[Lemma 2.1]{HaraTakagiOnAGeneralizationOfTestIdeals}, \cite[Proof of 3.18]{SchwedeTestIdealsInNonQGor}}]
If $X$ is a normal $F$-finite scheme, $\Delta \geq 0$ is a $\bQ$-divisor on $X$, $\fra$ is an ideal sheaf and $t \geq 0$ is a real number then for any sufficiently large effective Weil divisor $C$, we define the sheaf
\[ \tau(X, \Delta, \fra^t) := \sum_{e \geq 0} \, \sum_{\phi \in \sC_{\Delta}^{e}} \phi \Big( F^e_* \big( \fra^{\lceil t(p^e-1)\rceil} \cdot \O_X(-C) \big) \Big) \]
where $\sC_{\Delta}^e = \sHom_{\O_X} \big( F^e_* \O_X \big( \lceil (p^e - 1) \Delta \rceil \big), \O_X \big)$.
\end{definition}

\begin{remark} \label{rem.GlobalDefnTau}
The choice of $C$ is philosophically the same as the choice of a test element usually included in the (local) literature mentioned above.
Note that for any affine chart $U = \Spec R \subseteq X$, if $c \in R$ is an appropriate test element and $C$ is a Weil divisor on $X$ such that $\O_X(-C)\big|_U \subseteq c \cdot \O_U$, then one can always find another test element $d \in R$ with $d \cdot \O_U \subseteq \O_X(-C)\big|_U$.
It follows that our definition of the test ideal is indeed independent of the choice of $C$.
\end{remark}

\begin{lemma}
Notation as above. Then for any $e_0 \ge 0$, we have
\[ \tau(X, \Delta, \fra^t) = \sum_{e \ge e_0} \, \sum_{\phi \in \sC_{\Delta}^{e}} \phi \Big( F^e_* \big( \fra^{\lceil t(p^e-1)\rceil} \cdot \O_X(-C) \big) \Big). \]
\end{lemma}

\begin{proof}
The inclusion ``$\supseteq$'' is clear.
For ``$\subseteq$'', choose $C_1$ a sufficiently large Cartier divisor such that $\O_X(-C_1)$ is contained in the right-hand side, and put $C' = C + p^{e_0 - 1} C_1$.
Then we see that
\[ \sum_{e \ge 0} \, \sum_{\phi \in \sC_{\Delta}^{e}} \phi \Big( F^e_* \big( \fra^{\lceil t(p^e-1)\rceil} \cdot \O_X(-C') \big) \Big) \subseteq
\sum_{e \ge e_0} \, \sum_{\phi \in \sC_{\Delta}^{e}} \phi \Big( F^e_* \big( \fra^{\lceil t(p^e-1)\rceil} \cdot \O_X(-C) \big) \Big) \]
since we can split up the sum on the left-hand side as $\sum_{e=0}^{e_0-1} (\,\cdots) + \sum_{e \ge e_0} (\,\cdots)$.
But the left-hand side is just $\tau(X, \Delta, \fra^t)$, hence the lemma is proved.
\end{proof}

If $C$ is in fact Cartier, an easy direct computation yields
\begin{align*}
& \sum\nolimits_{\phi \in \sC_{\Delta}^{e}} \phi \Big( F^e_* \big( \fra^{\lceil t(p^e-1)\rceil} \cdot \O_X(-C) \big) \Big) \\
= \; & \Image \Big[ \big( F^e_* \fra^{\lceil t(p^e - 1) \rceil} \big) \cdot \sHom_{\O_X} \big( F^e_* \O_X \big( \lceil (p^e - 1)\Delta\rceil + C \big), \O_X \big) \xrightarrow{\textnormal{eval@1}} \O_X \Big],
\end{align*}%
hence
\begin{align*}
& \tau(X, \Delta, \fra^t) \\
= \; & \sum\nolimits_{e \ge e_0} \Image \Big[ \big( F^e_* \fra^{\lceil t(p^e - 1) \rceil} \big) \cdot \sHom_{\O_X} \big( F^e_* \O_X \big( \lceil (p^e - 1)\Delta\rceil + C \big), \O_X \big) \xrightarrow{\textnormal{eval@1}} \O_X \Big] \\
= \; & \sum\nolimits_{e \ge e_0} \Image \Big[ F^e_* \Big( \fra^{\lceil t(p^e - 1) \rceil} \cdot \O_X \big( (1 - p^e) K_X - \lceil (p^e - 1) \Delta \rceil - C \big) \Big) \xrightarrow{\;\;\tr\;\;} \O_X \Big] \\
= \; & \sum\nolimits_{e \ge e_0} \Image \Big[ F^e_* \Big( \fra^{\lceil t(p^e - 1) \rceil} \cdot \O_X \big( \rd (1 - p^e) (K_X + \Delta). - C \big) \Big) \xrightarrow{\;\;\tr\;\;} \O_X \Big]. \\
\end{align*}%
However, for our purposes this definition of the test ideal is not quite optimal.
Fortunately, this is easy after adjusting $C$.

\begin{lemma} \label{lem.test ideal}
With notation as above, assume that $C$ is Cartier. Then for any $e_0 \ge 0$ we have
\[ \tau(X, \Delta, \fra^t) = \sum_{e \ge e_0} \Image \Big[ F^e_* \Big( \fra^{\lceil t p^e \rceil} \cdot \O_X \big( \rd (1 - p^e) (K_X + \Delta). - C \big) \Big) \xrightarrow{\;\;\tr\;\;} \O_X \Big]. \]
\end{lemma}

\begin{proof}
Without loss of generality, and in view of \autoref{rem.GlobalDefnTau}, we may assume that $X = \Spec R$ is affine and that $\O_X(-C) = c \cdot \O_X$ for some $c \in R$.
Choose $b \in \fra^{\lceil t \rceil}$ nonzero, so that $b \fra^{\lceil t(p^e - 1) \rceil} \subseteq \fra^{\lceil t \rceil}\fra^{\lceil t(p^e - 1) \rceil} \subseteq \fra^{\lceil t p^e \rceil}$.
Replacing $C = \Div(c)$ by $\Div(bc)$ we obtain our desired formula.
\end{proof}

\subsection{Multiplier ideals}

In this section we work over a field of characteristic zero.
The theory of non-$\bQ$-Gorenstein multiplier ideals was developed by~\cite{DeFernexHaconSingsOnNormal}.
The starting point is a notion of pullback for Weil divisors.

\begin{definition}[\protect{\cite[Def.~2.6]{DeFernexHaconSingsOnNormal}}]
Let $f\!: Y \to X$ be a proper birational morphism between normal varieties, and let $D$ be an integral Weil divisor on $X$. Then the \emph{natural pullback} $f^\ntl D$ of $D$ along $f$ is defined by
\[ \O_Y(-f^\ntl D) = \big( \O_X(-D) \cdot \O_Y \big)^{**}, \]
where we consider $\O_X(-D) \subset \mathcal K_X$ as a fractional ideal sheaf on $X$.
\end{definition}

We will consider \emph{triples} $(X, \Delta, \fra^t)$ consisting of a normal variety $X$, an effective $\bQ$-divisor $\Delta \ge 0$, a nonzero coherent ideal sheaf $\fra \subset \O_X$ and a real number $t \ge 0$. In the case $\Delta = 0$, the following definition was made in~\cite[Definition~4.8]{DeFernexHaconSingsOnNormal}.

\begin{definition}[\protect{\cite[Def.~2.19]{ChiecchioEnescuMillerSchwede}}]
Let $(X, \Delta, \fra^t)$ be a triple, and let $m \in \bN$ be a positive integer such that $m\Delta$ is integral.
Let $f\colon Y \to X$ be a log resolution of the pair $\big( X, \O_X(-m(K_X + \Delta)) + \fra \big)$ in the sense of~\cite[Definition~4.1]{DeFernexHaconSingsOnNormal}.
Let $Z$ be the Cartier divisor on $Y$ such that $\fra \cdot \O_Y = \O_Y(-Z)$. Then we define
\[ \mJ_m(X, \Delta, \fra^t) := f_* \O_Y \Big( \big\lceil K_Y - \textstyle\frac 1m f^\ntl \big( m(K_X + \Delta) \big) - tZ \big\rceil \Big). \]
\end{definition}

One shows that this is a coherent ideal sheaf on $X$, independent of the choice of the resolution $f$.
Furthermore, $\mJ_m(X, \Delta, \fra^t) \subset \mJ_{km}(X, \Delta, \fra^t)$ for any integer $k > 0$.
Thus by the Noetherian property of $X$, the following definition makes sense.

\begin{definition}
The multiplier ideal $\mJ(X, \Delta, \fra^t)$ of a triple as above is defined to be the unique maximal element of the family
\[ \big\{ \mJ_m(X, \Delta, \fra^t) \mid \text{$m \ge 1$ and $m\Delta$ is integral} \big\}, \]
i.e.~it is equal to $\mJ_m(X, \Delta, \fra^t)$ for $m$ sufficiently divisible.
\end{definition}

We will need the following notion of \emph{compatible boundaries}, which is a straightforward generalization of the $\Delta = 0$ case in~\cite[Definition~5.1]{DeFernexHaconSingsOnNormal}.

\begin{definition} \label{def.compatible}
Let $(X, \Delta, \fra^t)$ be a triple, and fix an integer $m \ge 2$ such that $m\Delta$ is integral.
Given a log resolution $f\colon Y \to X$ of $\big( X, \O_X(-m(K_X + \Delta)) + \fra \big)$, a $\bQ$-Weil divisor $\Delta'$ on $X$ is called \emph{$m$-compatible} for $(X, \Delta, \fra^t)$ with respect to $f$ if the following hold:
\begin{enumerate}
\item[(i)]   $K_X + \Delta + \Delta'$ is $\bQ$-Cartier, \\[-.75em]
\item[(ii)]  $m\Delta'$ is integral and $\rd \Delta'. = 0$, \\[-.75em]
\item[(iii)] no component of $\Delta'$ is contained in $\supp(\Delta) \cup \supp(\O_X/\fra)$, \\[-.75em]
\item[(iv)]  $f$ is a log resolution of $\big( (X, \Delta + \Delta'), \O_X(-m(K_X + \Delta)) + \fra \big)$, \\[-.75em]
\item[(v)]   $K_Y + f^{-1}_* \Delta' - f^*(K_X + \Delta + \Delta') = K_Y - \textstyle\frac 1m f^\ntl \big( m(K_X + \Delta) \big)$.
\end{enumerate}
\end{definition}

\begin{proposition}
Let $(X, \Delta, \fra^t)$ be a triple, and fix an integer $m \ge 2$ such that $m\Delta$ is integral and $\mJ(X, \Delta, \fra^t) = \mJ_m(X, \Delta, \fra^t)$.
Then for any $m$-compatible boundary $\Delta'$ we have
\[ \mJ(X, \Delta, \fra^t) = \mJ \big( (X, \Delta + \Delta'); \fra^t \big), \]
where the right-hand side is a multiplier ideal in the usual $\bQ$-Gorenstein sense~\cite[Definition~9.3.60]{LazarsfeldPositivity2}.
\end{proposition}

\begin{proof}
The proof is analogous to~\cite[Proposition~5.2]{DeFernexHaconSingsOnNormal}, and thus it is omitted.
\end{proof}

The existence of compatible boundaries is ensured by the following theorem, cf.~\cite[Theorem~5.4]{DeFernexHaconSingsOnNormal} and~\cite[Theorem~4.4]{GrafJumpingCoefficientsOfNonQGor}.
The \emph{Weil index} of a triple $(X, \Delta, \fra^t)$ is defined to be the smallest positive integer $m_0$ such that $m_0(K_X + \Delta)$ is integral.

\begin{theorem} \label{thm.compatible}
Let $(X, \Delta, \fra^t)$ be a triple of Weil index $m_0$, and let $k \ge 2$ be an integer.
Choose an effective Weil divisor $D$ on $X$ such that $m_0(K_X + \Delta) - D$ is Cartier, and let $\sL \in \Pic X$ be a line bundle such that $\sL(-kD) := \sL \tensor \O_X(-kD)$ is globally generated. Pick a finite-dimensional subspace $V \subset H^0 \big( X, \sL(-kD) \big)$ that generates $\sL(-kD)$, and let $M$ be the divisor of a general element of $V$. Then
\[ \Delta' := \frac{1}{km_0} M \]
is a $(km_0)$-compatible boundary for $(X, \Delta, \fra^t)$.
\end{theorem}

\begin{proof}
Let $f\colon Y \to X$ be a log resolution of $\big( (X, \Delta), \O_X(-km_0(K_X + \Delta)) + \O_X(-kD) + \fra \big)$, and set $E := f^\ntl(kD)$. Then
\[ f^\ntl \big( km_0(K_X + \Delta) \big) = f^\ntl \big( km_0(K_X + \Delta) - kD + kD \big) = km_0 \cdot f^* \big( K_X + \Delta - \textstyle\frac1{m_0} D \big) + E \]
and so we have
\[ K_Y - \textstyle\frac1{km_0} f^\ntl \big( km_0(K_X + \Delta) \big) = K_Y - f^* \big( K_X + \Delta - \frac1{m_0} D \big) - \frac1{km_0} E. \]
Let $M$ be the divisor of a general element of $V$. Then since $\sL(-kD)$ is generated by $V$, we see that $M$ is reduced with no component contained in $\supp(\Delta) \cup \supp(\O_X/\fra)$.
Put $G = M + kD$, a Cartier divisor. Since also $f^* \sL \tensor \O_Y(-E)$ is generated by the (pullbacks of the) sections in $V$, we have that $f^* G = f^{-1}_* M + E$.

Now set $\Delta' = \frac1{km_0} M$. Then by the above, conditions (ii)--(iv) of \autoref{def.compatible} are satisfied. Also it is clear that
\[ K_X + \Delta + \Delta' = K_X + \Delta - \textstyle\frac1{m_0} D + \frac1{km_0} G \]
is $\bQ$-Cartier, so (i) is fulfilled. To check (v), note that
\begin{align*}
& K_Y + f^{-1}_* \Delta' - f^*(K_X + \Delta + \Delta') \\
= \; & K_Y + f^{-1}_* \Delta' - f^* \big( K_X + \Delta + \Delta' - \textstyle\frac1{km_0} G \big) - \frac1{km_0} f^* G \\
= \; & K_Y - f^* \big( K_X + \Delta - \textstyle\frac1{m_0} D \big) - \frac1{km_0} E \\
= \; & K_Y - \textstyle\frac 1{km_0} f^\ntl \big( km_0(K_X + \Delta) \big).
\end{align*}%
This proves the theorem.
\end{proof}

\section{Global generation at isolated non-finitely generated points}

The following theorem is a positive characteristic version of~\cite[Theorem~7.1]{GrafJumpingCoefficientsOfNonQGor}.
But even in characteristic zero (cf.~\autoref{rem.Char0Remark}), the present result is stronger than that theorem.
Most notably, we remove the divisibility condition in~\cite[Theorem 7.1]{GrafJumpingCoefficientsOfNonQGor} and replace the ``$K_X + \Delta$ is $\bQ$-Cartier'' condition with the weaker requirement that $\sR \big( X, -(K_X + \Delta) \big)$ be finitely generated.

\begin{theorem} \label{thm.AsymptoticGG}
Let $X$ be a normal projective $d$-dimensional variety over an $F$-finite field $k$ of characteristic $p > 0$.
Further let $D$ be a Weil $\bQ$-divisor and $C$ a Cartier divisor on $X$.
Suppose that $W \subseteq X$ is a closed set such that $\sR(X, D) := \bigoplus_{m \geq 0} \O_X(\lfloor mD \rfloor)$ is finitely generated on $X \setminus W$, let $W_0$ be the set of isolated points of $W$ and put
\[ U = (X \setminus W) \cup W_0. \]
Then there exists an ample Cartier divisor $H$ on $X$ such that for all $e \ge 0$, $m \ge 0$, $\ell \ge \max \{ m, p^e \}$ and any nef Cartier divisor $N$, the sheaf
\[ F^e_* \O_X \big( \rd mD + \ell H. - C + N \big) \]
is globally generated on $U$ (as an $\O_X$-module).
Furthermore, for any ample Cartier divisor $H'$ on $X$ fixed in advance, $H$ can be taken to be a sufficiently high multiple of $H'$.
\end{theorem}

\begin{remark} \label{rem.Char0Remark}
\autoref{thm.AsymptoticGG} continues to hold over an arbitrary field $k$ of characteristic zero if one interprets $F = \id_X$ and $p^e = 1$.
The proof does not require any changes.
\end{remark}

\begin{proof}[Proof of~\autoref{thm.AsymptoticGG}]
The strategy is similar to \cite[Theorem 7.1]{GrafJumpingCoefficientsOfNonQGor}.
We will find a globally generated sheaf $F^e_* \sF_{m, \ell} \into F^e_* \O_X \big( \rd mD + \ell H. - C + N \big)$ so that the cokernel is supported on $W$.
The proof is divided into three steps.

\subsubsection*{Step 1: Blowing up}

Let $N_0$ be a positive integer such that $N_0 D$ is an integral Weil divisor. It follows that the Veronese subring $\sR(X \setminus W, N_0 D)$ is also Noetherian and $\sR(X \setminus W, D)$ is a finite $\sR(X \setminus W, N_0 D)$-module~\cite[Lemma 2.4]{GotoHerrmannNishidaVillamayor}. By~\cite[Theorem 3.2(3)]{GotoHerrmannNishidaVillamayor}, making $N_0$ more divisible if necessary we may assume that $\sR(X \setminus W, N_0 D)$ is generated in degree $1$ as a graded ring.

Let $f\colon Y \to X$ be the normalized blowup of the fractional ideal sheaf $\O_X(N_0 D)$.
Then we have $f^{-1}(X \setminus W) = \Proj \sR(X \setminus W, N_0 D)$.
In particular, $f$ is a small morphism over $X \setminus W$ by~\cite[Lemma 6.2]{KollarMori}.
Furthermore if we write
\[ \O_Y(B) = \factor{f^* \big( \O_X(N_0 D) \big)}{\mathrm{torsion}} = \O_X(N_0 D) \cdot \O_Y, \]
then $B$ is Cartier and $f$-ample by~\cite[Theorem~6.2]{GrafJumpingCoefficientsOfNonQGor}.
Thus by~\cite[II, Proposition 7.10]{Hartshorne} or \cite[Prop.~1.45]{KollarMori} there exists a very ample Cartier divisor $A$ on $X$ so that $B + f^* A$ is globally ample on $Y$.

\subsubsection*{Step 2: Vanishing}

Now for any integer $m \in \N_0$, write uniquely
\[ m = q_m N_0 + r_m \quad \text{with $0 \leq r_m \leq N_0 - 1$,} \]
i.e.~$q_m = \rd m / N_0.$.
Fix a nef Cartier divisor $N$ on $X$ and form the sheaf
\[ \sG_m := \O_Y(q_m B + f^* N) \otimes \underbrace{\big( \O_X \big( \lfloor r_m D - C \rfloor \big) \cdot \O_Y \big)}_{=: \sH_m}. \]

\begin{claim} \label{clm:FujitaSerre}
There exists an $m_0 \ge 1$ such that for all $\ell \ge m \ge m_0$, $i \geq 1$, and any nef Cartier divisor $P$ on $Y$ we have
\begin{equation} \label{eq:Fujita}
\phantom{\text{and} \quad}
H^i \big( Y, \sG_m \otimes \O_Y(\ell f^* A  + P) \big) = 0 \quad \text{and}
\end{equation}
\begin{equation} \label{eq:relSerre}
f_* \sG_m \qis \myR f_* \sG_m.
\end{equation}
\end{claim}

\begin{proof}[Proof of~\autoref{clm:FujitaSerre}]
Let us make two easy observations: Firstly, the sheaf $\sH_m$ can take on only finitely many values. Secondly, $q_m \to \infty$ as $m \to \infty$.
Hence~\autoref{eq:Fujita} follows from Fujita vanishing~\cite[Theorem 1]{FujitaVanishingTheoremsForSemiPositive} applied to the ample divisor $B + f^* A$ on $Y$ upon writing
\begin{align*} \label{eq.FujitaVanishing}
& H^i \big( Y, \sG_m \otimes \O_Y(\ell f^* A  + P) \big) \\
= \; & H^i \big( Y, \sH_m \otimes \O_Y(q_m B + \ell f^* A + f^* N + P) \big) \\
= \; & H^i \big( Y, \sH_m \otimes \O_Y \big( q_m (B + f^* A) + \underbrace{(\ell - q_m)}_{\geq 0} f^* A + f^* N + P \big) \big).
\end{align*}%
Similarly, making $m_0$ even larger if necessary, \autoref{eq:relSerre} follows from relative Serre vanishing~\cite[Theorem~1.7.6]{LazarsfeldPositivity1} for the $f$-ample divisor $B$.
\end{proof}

\begin{claim} \label{clm:0reg}
Put $H = b \cdot A$ where $b > d = \dim X$. Then for every $e \ge 0$, $m \ge m_0$, and $\ell \ge \max \{ m, p^e \}$, the sheaf $F^e_* \big( f_* \sG_m \tensor \O_X(\ell H) \big)$ is $0$-regular with respect to $A$, and hence globally generated as an $\O_X$-module.
Furthermore its first cohomology group vanishes, $H^1 \big( X, F^e_* \big( f_* \sG_m \tensor \O_X(\ell H) \big) \big) = 0$.
\end{claim}

\begin{proof}[Proof of~\autoref{clm:0reg}]
We need to show that for every $1 \le j \le d$, we have
\[ H^j \big( X, F^e_* \big( f_* \sG_m \tensor \O_X(\ell H) \big) \tensor \O_X(-jA) \big) = 0. \]
The left-hand side is
\begin{align*}
 & = H^j \Big( X, F^e_* \big( f_* \sG_m \tensor \O_X \big( (\ell b - p^e j) A \big) \big) \Big) && \text{projection formula} \\
 & = H^j \big( X, f_* \sG_m \tensor \O_X \big( (\ell b - p^e j) A \big) \big) && \text{$F^e$ is finite\footnotemark} \\
 & = \bH^j \big( X, \myR f_* \sG_m \tensor \O_X \big( (\ell b - p^e j) A \big) \big) && \text{\autoref{eq:relSerre}} \\
 & = H^j \big( Y, \sG_m \tensor \O_Y \big( (\ell b - p^e j) f^* A \big) \big) && \text{composition of derived functors} \\
 & = 0 && \text{\autoref{eq:Fujita} and $\ell b - p^e j \ge m$.}
\end{align*}%
\footnotetext{ One may also argue by noting that $F^e_*(\,\cdot\,)$ leaves the abelian sheaf structure, and hence sheaf cohomology, unchanged.}%
The assertion $\ell b - p^e j \ge m$ in the last line is justified since
\[ \ell b \ge \ell + \ell d \ge m + p^e d \ge m + p^e j. \]
Any coherent sheaf $0$-regular with respect to $A$ is globally generated by Castelnuovo--Mumford regularity~\cite[Theorem~1.8.5]{LazarsfeldPositivity1}.
The desired vanishing
\[ H^1 \big( X, F^e_* \big( f_* \sG_m \tensor \O_X(\ell H) \big) \big) = 0 \]
follows by the same argument as above, leaving off the $\O_X(-jA)$.
\end{proof}

\subsubsection*{Step 3:  Reflexification and global generation}

\begin{claim} \label{clm:reflhull}
For any $e, m, \ell \ge 0$, the reflexive hull of $F^e_* \big( f_* \sG_m \tensor \O_X(\ell H) \big)$ is equal to
\[ F^e_* \O_X \big( \rd mD + \ell H. - C + N \big). \]
Furthermore the cotorsion of $F^e_* \big( f_* \sG_m \tensor \O_X(\ell H) \big)$ is supported on $W$, i.e.~in the natural short exact sequence
\[ 0 \to F^e_* \big( f_* \sG_m \tensor \O_X(\ell H) \big) \to F^e_* \O_X \big( \rd mD + \ell H. - C + N \big) \to \mathcal{Q}_{e, m, \ell} \to 0 \]
the support of $\mathcal{Q}_{e, m, \ell}$ is contained in $W$.
\end{claim}

\begin{proof}[Proof of~\autoref{clm:reflhull}]
Let $V \subset X$ be the maximal open subset over which $f$ is an isomorphism.
Note that $\codim_X(X \setminus V) \ge 2$. We see that
\[ (f_* \sG_m)\big|_V = \O_V(q_m N_0 D + N) \tensor \O_V \big( \rd r_m D - C. \big) = \O_V \big( \rd mD. - C + N \big), \]
the first equality holding by definition and the second one because $N_0 D$ is Cartier on $V$. Pushing this forward by the inclusion $V \into X$, we get
\begin{equation} \label{eq:reflhull}
(f_* \sG_m)^{**} = \O_X \big( \rd mD. - C + N \big).
\end{equation}
Observe that Frobenius pushforward commutes with taking the reflexive hull.
Hence twisting~\autoref{eq:reflhull} by $\O_X(\ell H)$ and applying $F^e_*(\,\cdot\,)$ proves the first part of the claim.
For the second part, use the fact that $f$ is small over $X \setminus W$ and that the pushforward of a reflexive sheaf under a small birational map is again reflexive.
\end{proof}

Returning to the proof of \autoref{thm.AsymptoticGG}, choose a point $x \in U = (X \setminus W) \cup W_0$, and assume first that $m \ge m_0$.
We see that $\mathcal Q = \mathcal{Q}_{e, m, \ell}$ is globally generated at $x$ since either
\begin{itemize}[leftmargin=*]
\item $x \in W_0$ and then by~\autoref{clm:reflhull}, $x$ is an isolated point of $\supp \mathcal{Q}$ or $\mathcal{Q}$ is even zero at $x$, or
\item $x \in X \setminus W$ and then $\mathcal Q$ definitely is zero at $x$.
\end{itemize}
Hence $F^e_* \O_X \big( \rd mD + \ell H. - C + N \big)$ is globally generated at $x$ by~\autoref{clm:0reg} and \cite[Lemma 7.3]{GrafJumpingCoefficientsOfNonQGor}.

To finish the proof, we still need to take care of the sheaves
\begin{equation} \label{eq:finitely many m}
F^e_* \O_X \big( \rd mD + \ell H. - C + N \big) \quad \text{for $e \ge 0$, $0 \le m < m_0$, and $\ell \ge \max \{ m, p^e \}$.}
\end{equation}
To this end, notice that arguing as in the proof of~\autoref{clm:0reg}, for $1 \le j \le d$ we have
\begin{align*}
& H^j \big( X, F^e_* \O_X \big( \rd mD + \ell H. - C + N \big) \tensor \O_X(-jA) \big) \\
= \; & H^j \big( X, \underbrace{\O_X \big( \rd mD. - C \big)}_{\text{finitely many values}} \tensor \O_X \big( \underbrace{(\ell b - p^e j)}_{\ge p^e(b-d)} A + N \big) \big).
\end{align*}%
Hence by Fujita vanishing, taking $b$ sufficiently large in~\autoref{clm:0reg}, we may assume that the sheaves~\autoref{eq:finitely many m} are $0$-regular with respect to $A$. In particular they are globally generated on $U$.

To justify the last claim of~\autoref{thm.AsymptoticGG}, simply note that for any ample Cartier divisor $H'$ on $X$ given in advance, we may pick $A$ to be a sufficiently high multiple of $H'$ and then also $H$ will be a multiple of $H'$.
\end{proof}

\section{Proof of main results}

In this section we prove the results announced in the introduction.
But first we state a weak result on global generation of test ideals.
Compare with \cite{MustataNonNefLocusPositiveChar,SchwedeACanonicalLinearSystem,KeelerFujita}.

\begin{proposition} \label{prop.UniformGlobalGenTau}
Suppose $X$ is a normal projective variety over an $F$-finite field $k$ of characteristic $p > 0$, $\Delta \geq 0$ is a $\bQ$-divisor, $\fra$ is a nonzero coherent ideal sheaf and $t_0 > 0$ is a real number.
Suppose that $\sR\big(X, -(K_X + \Delta)\big)$ is finitely generated away from a closed set $W \subseteq X$ and that $W_0 \subseteq W$ is the set of isolated points of $W$.
Set $U = (X \setminus W) \cup W_0$.  Then there exists an ample divisor $H$ such that
\[ \tau(X, \Delta, \fra^t) \otimes \O_X(H) \]
is globally generated on $U$ for all $t \in [0, t_0]$.
\end{proposition}

\begin{proof}
Choose an effective Cartier divisor $C \geq 0$ on $X$ so that $\O_X(-C) \subseteq \tau(X, \Delta, \fra^{t_0}) \subseteq \tau(X, \Delta, \fra^t)$.
By \autoref{lem.test ideal}, for any $t \in [0, t_0]$ we have
\begin{equation} \label{eq:test ideal}
\tau(X, \Delta, \fra^t) = \sum_{e \ge 0} \Image \Big[ F^e_* \Big( \fra^{\lceil t p^e \rceil} \O_X \big( \rd (1 - p^e) (K_X + \Delta). - C \big) \Big) \xrightarrow{\;\;\tr\;\;} \O_X \Big].
\end{equation}
Now fix an ample Cartier divisor $A$ on $X$ so that
\[
\fra^{\lceil t \rceil} \tensor \O_X(A)
\]
is globally generated for all $t \in [0, t_0]$.  We then observe that for all $m > 0$ and $t \in [0, t_0]$,
\begin{equation} \label{eq:twisted ideal}
\fra^{\rp mt.} \tensor \O_X(mA)
\end{equation}%
is also globally generated. The reason is that since $\rp mt. \le m \rp t.$, the ideal $\fra^{\lceil mt \rceil}$ can be written as the product of $m$ ideals of the form $\fra^{\rp s.}$ for various values of $s \in [0, t_0]$.
For ease of notation, write $W_m^t$ for the $k$-vector space of global sections of the sheaf~\autoref{eq:twisted ideal}.
It follows that for every $m > 0$, the map
\[
W_m^t \tensor_k \O_X(-mA) \to \fra^{\rp mt.}
\]
is surjective. Combining with~\autoref{eq:test ideal}, we get that
\begin{align*}
\tau(X, \Delta, \fra^t) & =
  \sum_{e \geq 0} \Img \Big[ F^e_* \Big( \fra^{\lceil tp^e \rceil} \O_X \big( \lfloor (1 - p^e)(K_X + \Delta) \rfloor - C \big) \Big) \to \O_X \Big] \\
& = \sum_{e \geq 0} \Img \Big[ F^e_* \Big( W_{p^e}^t \tensor_k \O_X(-p^e A) \tensor \O_X \big( \lfloor (1 - p^e)(K_X + \Delta) \rfloor - C \big) \Big) \to \O_X \Big] \\
& = \sum_{e \geq 0} \Img \Big[ F^e_* \Big( W_{p^e}^t \tensor_k \O_X \big( \lfloor (p^e - 1)(-K_X - \Delta - A) \rfloor - C - A \big) \Big) \to \O_X \Big].
\end{align*}
Now choose an ample divisor $H$ that satisfies the conclusion of \autoref{thm.AsymptoticGG} relative to the divisor $D = -(K_X + \Delta + A)$, where $C + A$ takes the role of $C$. Then for $m = p^e - 1$, $\ell = p^e$ and $N = 0$ we get that
\begin{equation} \label{eq:gg}
F^e_* \Big( \O_X \big( \lfloor (p^e - 1)(-K_X - \Delta - A) \rfloor - C - A \big) \Big) \tensor \O_X(H)
\end{equation}
is globally generated (as an $\O_X$-module) over $U$ for all $e \geq 0$.
Hence also
\[ \tau(X, \Delta, \fra^t) \tensor \O_X(H) \]
is globally generated over $U$, being a quotient of a direct sum of sheaves of the form~\autoref{eq:gg}. This completes the proof.
\end{proof}

\begin{theorem} \label{thm.discrete isol}
Suppose that $X$ is a normal variety over an $F$-finite field $k$ of positive characteristic and that $\Delta \geq 0$ is a $\bQ$-divisor such that $\sR \big( X, -(K_X + \Delta) \big)$ is finitely generated except at an isolated collection of points. Suppose $\fra \subseteq \O_X$ is a nonzero coherent ideal sheaf. Then the $F$-jumping numbers of $(X, \Delta, \fra)$ have no limit points.
\end{theorem}

\begin{proof}
By~\cite[Proposition~3.28]{BlickleSchwedeTakagiZhang}, we may assume that $X$ is affine.
Let $\overline{X}$ denote the closure of $X$ in some projective space.
By normalizing, we may also assume that $\overline{X}$ is normal.
There exists a $\bQ$-divisor $\wb\Delta \ge 0$ on $\wb X$ and a coherent ideal sheaf $\wb\fra \subset \O_{\wb X}$ which restrict to $\Delta$ and $\fra$, respectively.

Pick an arbitrary real number $t_0 > 0$.  By \autoref{prop.UniformGlobalGenTau}, we know that there exists an ample Cartier divisor $H$ on $\wb X$ such that
\[ \tau \big( \wb X, \wb\Delta, \wb\fra^t \big) \otimes \O_{\wb X}(H) \quad \text{is globally generated on $X \subset \wb X$ for all $t \in [0, t_0]$.} \]
Note that $\tau \big( \wb X, \wb\Delta, \wb\fra^t \big) \big|_X = \tau(X, \Delta, \fra^t)$ since $X \subset \wb X$ is open.
Now it follows from~\cite[Lemma~8.2]{GrafJumpingCoefficientsOfNonQGor} that for any strictly increasing sequence of numbers $0 \le s_0 < s_1 < \cdots < t_0$, the corresponding sequence of test ideals
\[ \tau(X, \Delta, \fra^{s_0}) \supset \tau(X, \Delta, \fra^{s_1}) \supset \cdots \]
stabilizes. Hence the set of $F$-jumping numbers of $(X, \Delta, \fra)$ does not have a limit point in the interval $[0, t_0]$.
As $t_0 > 0$ was chosen arbitrarily, this proves the theorem.
\end{proof}

\begin{theorem} \label{thm.discrete klt}
Suppose that $X$ is a normal variety over a field $k$ of characteristic zero and that $\Delta \geq 0$ is a $\bQ$-divisor such that $\sR \big( X, -(K_X + \Delta) \big)$ is finitely generated except at an isolated collection of points. Suppose $\fra \subseteq \O_X$ is a nonzero coherent ideal sheaf. Then the jumping numbers of $(X, \Delta, \fra)$ have no limit points.
\end{theorem}

\begin{proof}
The proof follows quite closely along the lines of~\cite[Theorem~8.1]{GrafJumpingCoefficientsOfNonQGor}.
For the reader's convenience, we give a sketch of the argument here.

Arguing by contradiction, assume that there is a strictly increasing and bounded above sequence $0 \le s_0 < s_1 < \cdots$ of jumping numbers of $(X, \Delta, \fra)$.
As above, we may assume that there is a triple $(\wb X, \wb\Delta, \wb\fra)$ containing $(X, \Delta, \fra)$ as an open subset and such that $\wb X$ is normal and projective.
Let $m_0$ be the Weil index of $(X, \Delta, \fra)$.
By \autoref{thm.AsymptoticGG} in combination with \autoref{rem.Char0Remark}, using \autoref{thm.compatible} we can construct for each $k \ge 2$ a $\bQ$-Weil divisor $\wb\Delta_k$ on $\wb X$ such that $\Delta_k := \wb\Delta_k\big|_X$ is a $(km_0)$-compatible boundary for $(X, \Delta, \fra)$ and furthermore the $\bQ$-linear equivalence class of $\wb\Delta_k$ does not depend on $k$.
The last property is crucial, as it enables us to find an ample Cartier divisor $H$ on $\wb X$ such that
\[ \mJ(\wb X, \wb\Delta + \wb\Delta_k, \wb\fra^{s_\ell}) \tensor \O_{\wb X}(H) \quad \text{is globally generated for all $k \ge 2$, $\ell \ge 0$,} \]
using~\cite[Proposition~8.3]{GrafJumpingCoefficientsOfNonQGor}.
Since $\Delta_k$ is $(km_0)$-compatible, it follows that $\mJ(\wb X, \wb\Delta, \wb\fra^{s_\ell}) \tensor \O_{\wb X}(H)$ is globally generated on $X$ for all $\ell \ge 0$.
By~\cite[Lemma~8.2]{GrafJumpingCoefficientsOfNonQGor}, this implies that the sequence of ideals
\[ \mJ(\wb X, \wb\Delta, \wb\fra^{s_0}) \supset \mJ(\wb X, \wb\Delta, \wb\fra^{s_1}) \supset \cdots \]
stabilizes when restricted to $X$.
Since $\mJ(\wb X, \wb\Delta, \wb\fra^{s_\ell})\big|_X = \mJ(X, \Delta, \fra^{s_\ell})$, this contradicts the assumption that each $s_i$ is a jumping number of $(X, \Delta, \fra)$.
\end{proof}

\begin{theorem} \label{thm.HSLG}
Suppose that $X$ is a normal variety over an $F$-finite field $k$ of characteristic $p > 0$.  Set
\[
J_e := J_e^X := \Image \Big(F^e_* \O_X \big( (1-p^e)K_X \big) \cong \sHom_{\O_X}(F^e_* \O_X, \O_X) \xrightarrow{\textnormal{eval@1}} \O_X \Big) \subset \O_X.
\]
If $\sR(X, -K_X)$ is finitely generated except at an isolated collection of points, then $J_e = J_{e+1}$ for all $e \gg 0$.
\end{theorem}

\begin{proof}
First notice that $J_e \supseteq J_{e+1}$ since $F^e_* \O_X \hookrightarrow F^{e+1}_* \O_X$ and $\sHom_{\O_X}(-, \O_X)$ is contravariant.
As in the proof of \autoref{thm.discrete isol}, we may assume that $X$ is an open subset of some normal projective variety $\wb X$.
By \autoref{thm.AsymptoticGG}, we know there exists an ample divisor $H$ on $\wb{X}$ such that
\[ F^e_* \O_{\wb X } \big( (1 - p^e) K_{\wb X} + p^e H \big) = F^e_* \O_{\wb X } \big( (1 - p^e) K_{\wb X} \big) \tensor \O_{\wb X}(H) \]
is globally generated on $X$, as an $\O_{\wb X}$-module, for all $e \ge 0$.
Hence its image $J_e^{\wb X} \tensor \O_{\wb X}(H)$ is also globally generated on $X$.
But $J_e^{\wb X}\big|_X = J_e^X$ since $X \subset \wb X$ is open.
Hence we see that $J_e = J_{e+1}$ for $e \gg 0$ by~\cite[Lemma~8.2]{GrafJumpingCoefficientsOfNonQGor}. This completes the proof.
\end{proof}

Finally we prove the final statement from the introduction.

\begin{proof}[\protect{Proof of~\autoref{prp.klt fg}}]
For~\ref{itm.klt fg.1}, we need to prove that for every characteristic zero klt pair $(X, D)$ and for every $\bQ$-divisor $B$ on $X$, the algebra $\sR(X, B)$ is finitely generated.
This is well-known to experts (see e.g.~\cite[Theorem~92]{KollarExercisesInBiratGeom}), but for completeness' sake we provide a proof.

The question is local, so we may assume that $B$ is effective and that $K_X + D \sim_\Q 0$.
Let $\pi\colon Y \to X$ be a small $\bQ$-factorial modification, which exists by~\cite[Corollary~1.4.3]{BirkarCasciniHaconMcKernan}.
For some rational $0 < \varepsilon \ll 1$, the pair $\big( Y, \pi^{-1}_*(D + \varepsilon B) \big)$ is klt.
The map $\pi$ being small, we have
\[ \pi_* \big[ \sR \big( Y, K_Y + \pi^{-1}_*(D + \varepsilon B) \big) \big] = \sR(X, K_X + D + \varepsilon B). \]
By~\cite[Theorem~1.2(3)]{BirkarCasciniHaconMcKernan}, the left-hand side is finitely generated. Hence so is the right-hand side.
Since $K_X + D \sim_\Q 0$ and $\varepsilon \in \Q$, we see that $\sR(X, K_X + D + \varepsilon B)$ and $\sR(X, B)$ have isomorphic Veronese subalgebras.
We conclude by~\cite[Lemma~2.4 and Theorem~3.2]{GotoHerrmannNishidaVillamayor}.

Concerning~\ref{itm.klt fg.3}, after shrinking $X$ we may assume that $X = \Spec R$ is affine and has pseudorational singularities.
If the singular locus of $X$ is zero-dimensional, we are clearly done.
So let $\frp \in \Spec R$ be the generic point of a one-dimensional component of $\Sing(X)$.
Localizing at $\frp$, we obtain a two-dimensional pseudorational germ $U := \Spec R_\frp \to X$ with closed point $\frm := \frp R_\frp$.
Let $\pi\colon \wt U \to U$ be a desingularization of $U$, with exceptional divisor $E \subset \wt U$.
The Grothendieck spectral sequence associated to the composition of functors $\Gamma_\frm \circ \pi_* = \Gamma_E$ yields an exact sequence
\[ \underbrace{\loccoh1.E.\wt U.\O_{\wt U}.}_{=0} \to \underbrace{\loccoh0.\frm.U.R^1 \pi_* \O_{\wt U}.}_{=\coh1.\wt U.\O_{\wt U}.} \to \underbrace{\loccoh2.\frm.U.\O_U. \to \loccoh2.E.\wt U.\O_{\wt U}.}_{\text{injective by pseudorationality}}, \]
where the first term is zero due to~\cite[Theorem~2.4]{LipmanDesingularizationOf2Dimensional}.
It follows that $\coh1.\wt U.\O_{\wt U}. = 0$, so $U$ has rational singularities in the sense of Lipman~\cite[Definition~1.1]{LipmanRationalSingularities}.

Now consider a $\bQ$-Weil divisor $B$ on $X$. By~\cite[Proposition~17.1]{LipmanRationalSingularities}, the restriction of $B$ to $U$ is $\bQ$-Cartier and then $B$ itself is $\bQ$-Cartier in a neighborhood of $\frp \in X$.
Applying this argument to every one-dimensional component of $\Sing(X)$, we see that except at an isolated collection of points, $B$ is $\bQ$-Cartier and in particular $\sR(X, B)$ is finitely generated.
\end{proof}

\def\cfudot#1{\ifmmode\setbox7\hbox{$\accent"5E#1$}\else
  \setbox7\hbox{\accent"5E#1}\penalty 10000\relax\fi\raise 1\ht7
  \hbox{\raise.1ex\hbox to 1\wd7{\hss.\hss}}\penalty 10000 \hskip-1\wd7\penalty
  10000\box7}
\providecommand{\bysame}{\leavevmode\hbox to3em{\hrulefill}\thinspace}
\providecommand{\MR}{\relax\ifhmode\unskip\space\fi MR}
\providecommand{\MRhref}[2]{%
  \href{http://www.ams.org/mathscinet-getitem?mr=#1}{#2}
}
\providecommand{\href}[2]{#2}

\end{document}